\documentclass[11pt]{amsart}

\usepackage{a4wide}
\usepackage{amsmath,amssymb,amsthm,amsfonts}
\usepackage{amsrefs}
\usepackage[utf8]{inputenc}
\usepackage[T1]{fontenc}
\usepackage{xcolor}

\def\R {\mathbb{R}}

\renewcommand{\phi}{\varphi}
\renewcommand{\epsilon}{\varepsilon}

\renewcommand{\leq}{\leqslant}
\renewcommand{\le}{\leqslant}
\renewcommand{\geq}{\geqslant}

\renewcommand{\div}{\mathrm{div}}

\newcommand{\Per}{\mathrm{Per}}

\newtheorem{proposition}{Proposition}[section]
\newtheorem{theorem}[proposition]{Theorem}

\newtheorem{lemma}[proposition]{Lemma}
\theoremstyle{definition}
\newtheorem{definition}[proposition]{Definition}

\newtheorem{remark}[proposition]{Remark}
\numberwithin{equation}{section}

\title[Graphical translators]{Graphical translators for anisotropic and crystalline mean curvature flow} 
\email{annalisa.cesaroni@unipd.it}
\email{heiko.kroener@uni-due.de}
\email{matteo.novaga@unipi.it}
\thanks{The authors were supported by the INDAM-GNAMPA and by the PRIN Project 2019/24 {\it Variational methods for stationary and evolution problems with singularities and interfaces}.}

\begin{document}
\maketitle

\centerline{\scshape Annalisa Cesaroni }
\medskip
{\footnotesize
%% please put the address of the first author
\centerline{Department of Statistical Sciences, University of Padova}
\centerline{Via Cesare Battisti 141, 35121 Padova, Italy}
} 

\medskip

\centerline{\scshape Heiko Kr\"oner}
\medskip
{\footnotesize
%%% please put the address of the first author
\centerline{Universit\"at Duisburg-Essen, Fakult\"at f\"ur Mathematik}
\centerline{Thea-Leymann-Stra\ss e 9, 45127, Essen, Germany}
}
 
\medskip

\centerline{\scshape Matteo Novaga}
\medskip
{\footnotesize
% % please put the address of the second  and third author
\centerline{Department of Mathematics, University of Pisa}
\centerline{Largo Bruno Pontecorvo 5, 56127 Pisa, Italy}
}

\begin{abstract} 
In this paper we discuss existence, uniqueness and some properties of a class of solitons to the anisotropic mean curvature flow, i.e., graphical translators, either in the plane or under an assumption of cylindrical symmetry on the anisotropy and  the mobility. In these cases, the equation becomes an ordinary differential equation, and this allows to find explicitly the translators and describe their main features.   
\end{abstract} 

\tableofcontents

\section{Introduction} 
 
We consider the evolution of sets $t\mapsto E_t$ in $\R^{N+1}$ governed by the geometric law 
 \begin{equation}\label{mcf} \partial_t p\cdot \nu(p)= -\psi(\nu(p))H_\phi(p, E_t),\end{equation} 
 where $\nu(p)$ is the  exterior normal  at
 $p\in \partial E_t$, $\psi$ is  a norm representing the mobility,  
 $\phi$ is a norm  representing the surface tension, and $H_\phi(p)$  is the anisotropic mean curvature of 
 $\partial E_t$ at $p$, see Definition \ref{defanisotropy}. This evolution is the gradient flow for the anisotropic perimeter
 $\int_{\partial E} \phi(\nu)d{\mathcal H}^N(y)$ with respect to a weigthed $L^2$ norm (depending on  $\psi$) and it is an analogue of the classical (isotropic) mean curvature flow, 
 which corresponds to the case $\phi(x)=\psi(x)=|x|$.      
 
 In this paper we are interested in a particular class of solutions to \eqref{mcf}, which are the graphical translators. 
   \begin{definition}\label{deftrans}  
 An entire graphical translator  is a solution to \eqref{mcf}  given by $\partial E_t=\partial E_0+ct \ e_{N+1}$, where $c\in\R$ and $\partial E_0$ is the  graph of a function $u:\R^N\to \R$. In particular $E_0$ solves
\begin{equation}\label{trans0}c e_{N+1}\cdot \nu(p)= -\psi(\nu(p))H_\phi(p, E_0). \end{equation} 

A complete graphical  translator  is a solution to \eqref{mcf}  given by $\partial E_t=\partial E_0+ct \ e_{N+1}$, where $c\in\R$ and $\partial E_0$ is a complete hypersurface which solves \eqref{trans0} and  is the  graph of a function $u:\Omega\subseteq \R^N\to \R$, where $\Omega\subset\R^N$ is an open set.
 \end{definition}
 
 In the isotropic case $\phi(p)=\psi(p)=|p|$ translating solutions have been intensively studied, and there is a huge literature on the topic, since they arise as blow-up limits of type-II singularities of the mean curvature flow. 
In every dimension there exist complete translating graphs, and the first example is the so called grim reaper in $\R^2$. On the other hand
entire graphical  translators  in $\R^{N+1}$ exist  only for  $N>1$ (see \cites{hof,w}),  and one of the main examples is the bowl soliton, which is the unique (up to translations) convex and  radially symmetric   solution (see \cite{clutt}). 
Finally we recall that in \cite{sx} Spruck and Xiao showed that every graphical translator in $\R^3$ is convex and in \cite{w} Wang proved that  the bowl soliton is the only one, whereas in $\R^{N+1}$ for $N>2$ there are entire graphical translators which are  convex but not radially symmetric.  In \cite{hof} all complete translating  graphs in $\R^3$ have been classified.  
 
 In this paper we initiate the analysis of translating graphs for the anisotropic mean curvature flow, and in particular we are interested in the description of grim reapers  and bowl solitons.  Section  \ref{grim} is devoted to grim reapers in $\R^2$ and in higher dimension, whereas in Section \ref{bowl}  we assume that  both the anisotropy and the mobility have a cylindrical symmetry,  and we study existence and qualitative properties of bowl solitons. 
 %that is entire graphical translators (so in $\R^{N+1}$ with $N>1$) which have the same symmetries as the projected Wulff shape of the anisotropy on the horizontal space. 
 
We construct our solutions assuming first that  the anisotropy is regular, see assumption \eqref{phiregular}, and then we treat the general case by approximation, using the stability result obtained in \cite{cmnp2}. We also   discuss directly the construction and main properties of these  soliton solutions,  in the case of purely crystalline anisotropy, see Remarks \ref{remc1} and \ref{remc2}. 

Finally we recall that other soliton solutions for the anisotropic mean curvature flow  in the graphical setting, namely the expanding self-similar solutions, have been studied by the authors in \cite{ckn} (see also \cite{ggh} for a result in the case of crystalline curvature flow in the plane). 

\section{Definitions and preliminary results} 
We recall some definitions for anisotropies and related geometric flows (see for instance  \cite{bccn}).

\begin{definition}\label{defanisotropy}  
Let  $\phi:\R^{N+1}\to [0, +\infty)$ be  a positively $1$-homogeneous convex map,   such that $\phi(p)>0$ for all $p\neq 0$. We associate to the surface tension 
the anisotropy $\phi^0:\R^{N+1}\to [0+\infty)$ defined as $\phi^0(q):=\sup_{\phi(p)\leq 1} p\cdot q$, which is again convex and positively $1$-homogeneous. 
The anisotropic mean curvature of a set $E$ at a point $p\in \partial E$ is defined as
\[H_\phi(p, E)=\div_\tau(\nabla \phi(\nu(p))),  \]  when $\phi$ is regular,  where $\nu(p)$ is the exterior normal vector to $\partial E$ at $p$, and $\div_\tau$ is the tangential divergence, 
whereas  in the general case  it is defined using the subdifferential, 
\[ H_\phi(p,E)\in\div_\tau(\partial \phi(\nu(p))).\] 
\end{definition} 
 We define the  Wulff shape as  the convex compact set  \[W_{\phi^0} :=\{q\in\R^{N +1}\ | \phi^0(q) \leq 1\}.\] 
By using this definition, with some computation it is possible to check that (see \cite{bp}) 
\begin{equation}\label{wcu} H_\phi(RW_{\phi^0})=\frac{N}{R}.\end{equation}  
We consider the geometric evolution law  \eqref{mcf} under the following assumptions on  mobility:
 \begin{equation}\label{asspsi}
\psi:\R^{N+1}\to [0, +\infty) \text{ is positively $1$-homogeneous, convex and }    \psi(p)>0  \qquad \text{for all } p\neq 0.
 \end{equation}  
 
Some results will be  first obtained assuming the following regularity assumption on the anisotropy:
  \begin{equation}\label{phiregular}    \text{ $\phi\in C^{2}(\R^{N+1}\setminus \{0\})$ and $\phi^2$ is uniformly convex,} \end{equation} 
and then extended by approximation, since the level set solutions we  consider are stable with respect to locally uniform convergence (see Theorem \ref{ex} and \cite{cmnp2})

 \begin{remark}\upshape\label{norma} 
 We collect some useful properties of $\phi$, that will be useful in the following.
 First of all  for all $p, p_0\in \R^{N+1}$, by convexity we get  \[\phi(p)\geq \phi(p_0)+\partial \phi(p_0)\cdot(p-p_0)\] where $\partial \phi(p_0)$ is the subdifferential. 
  Moreover by  convexity  and  positive $1$-homogeneity,  for every $p_0\in \R^{N+1}$ we get
  \[\phi(p_0)=\partial \phi(p_0)\cdot p_0.\]
  Finally by positive $1$-homogeneity of $\phi$, for every $\lambda\in \R$, $\lambda\neq 0$, we have that 
  \[\partial \phi(\lambda p_0)=\frac{\lambda}{|\lambda|}\partial \phi (p_0). \]
 \end{remark}
 We recall the following result about well posedness of the flow \eqref{mcf}, in the level set sense. In particular we introduce a uniformly continuous function  $U_0:\R^{N+1}\to\R$ such that $E_0=\{p\in \R^{N+1}:\ U_0(p)\leq 0 \}$  and $\partial E_0=\{p\in \R^{N+1}:\ U_0(p)= 0 \}$,   and we consider    the following quasi-linear parabolic equation 
 \begin{equation}\label{pdelevel}
 \begin{cases} U_t-\psi(\nabla U)\div(\nabla\phi(\nabla U))=0\\U(p,0)=U_0(p).
 \end{cases} 
  \end{equation}

 Existence and uniqueness of the level set flow associated to \eqref{mcf} have been obtained  for general  mobilities $\psi$ and purely crystalline norms $\phi$ in \cites{gp1,gp2}, 
  in the viscosity setting, whereas the case of general norms $\phi$ with convex  mobilities $\psi$  has been treated in \cites{cmnp1, cmnp2},
  in the distributional setting. 

 \begin{theorem}\label{ex}  
There exists a unique  continuous solution $U$ to \eqref{pdelevel}. The solution is intended  in the distributional sense as in \cite{cmnp1}, 
and coincides with the locally uniform limit of viscosity solutions to \eqref{pdelevel} when $\phi$ is locally uniformly approximated by $\phi_n$ which satisfy \eqref{phiregular}, see \cite{cmnp2}.    Therefore the  level set flows defined as 
 \begin{eqnarray*} E^+_t&:=&\{p\in \R^{N+1}:\ U(p,t)\leq 0\}\\E^-_t&:=&\{p\in \R^{N+1}:\ U(p,t)< 0\} \end{eqnarray*}
 provide a solution (in the distributional level set sense   \cite{cmnp2}) to \eqref{mcf}.
 
 Moreover, if $U_0,V_0$ are two uniformly continuous functions such that $U_0\leq V_0$, then $U(p,t)\leq V(p,t)$ for all $t>0$ and $p\in \R^{N+1}$. 
 
  Moreover if $E_0$ is the subgraph of an entire Lipschitz function, that is 
 \begin{equation}\label{initial} \exists u_0:\R^N\to \R, \text{ Lipschitz continuous such that }E_0=\{(x,z)\in \R^{N+1}\ |\ z\leq u_0(x)\},
 \end{equation} 
  the level set flow satisfies $\overline E_t^-= E_t^+=\{(x,z)\in \R^{N+1}\ |\ z\leq u(x,t)\}$, where $u(x,t)$
 is a continuous function such that  
 \[|u(x,t)-u(y,s)|\leq \|\nabla u_0\|_\infty |x-y|+K\sqrt{|t-s|}\] for some  $K>0$ depending only on the Lipschitz constant $\|\nabla u_0\|_\infty$ of $u_0$. 
 
When $\phi$ is regular, that is, \eqref{phiregular} holds, then $u$ is the viscosity solution to
  \begin{equation}\label{pdelevel1}
 \begin{cases} u_t+\psi(-\nabla u,1)\div(\nabla_x\phi(-\nabla u,1))=0\\u(x,0)=u_0(x).
 \end{cases} 
  \end{equation}

 \end{theorem}

Note that the Definition \ref{deftrans} on $E_0$ to be a graphical translator reads as a condition on the function $u$, whose graph is $\partial E_0$: in particular 
$u(x)+ct$ has to solve, in appropriate sense, \eqref{pdelevel1}, which means that $u$ has to solve 
\begin{equation}\label{trans} -\div(\nabla_x\phi(-\nabla u,1))=\frac{c}{\psi(-\nabla u,1)}.\end{equation}  

In order to construct translating solutions,  it is sufficient to solve Equation \eqref{trans}. Note  that if we can solve the equation for $c=1$, then up to a suitable dilations we  obtain a solution  to \eqref{trans} for every $c\neq 0$.

 First of all, we observe  that in the case of regular anisotropies there are no globally  Lipschitz  translating solutions  (with $c\neq 0$). 

\begin{proposition}  Assume that \eqref{phiregular} holds. Let $E_0\subseteq \R^{N+1}$  be the subgraph of a    Lipschitz continuous function $u_0$, and assume  that  there exists $c$, for which $E_t=E_0+ct \ e_{N+1}$ is a solution to \eqref{mcf}. Then $c=0$ and  $E_0$ is a half-space. 
\end{proposition}
\begin{proof}  First of all we show that necessarily the  speed $c$ is equal to $0$.   Since $ E_0$ is the subgraph of a Lipschitz continuous function $u_0$,  
we can write $E_t$ as the subgraph of a function $u( \cdot, t)$. By assumption we get that $u(x,t)=u_0(x)+ct$ is the solution to \eqref{pdelevel1}.   
On the other hand, by  Theorem \ref{ex}, there holds that $|u(x,t)-u_0(x)|\leq K\sqrt{t}$ for a constant $K$,  which implies $c=0$. 

It follows that  $u_0$   solves, in the viscosity sense,  
 \[ \div(\nabla\phi(-\nabla u_0,1))=0\qquad\text{ for all $x\in \R^N$.   }\]
By elliptic regularity theory for viscosity solutions (see \cite{tru}), this implies   $u_0\in C^{1,\alpha}(\R^N)$ for every $\alpha<1$, and then 
by a bootstrap argument $u_0\in C^\infty(\R^N)$. 

Now, we observe that for every $i$,  $(u_0)_{x_i}$ is a bounded  solution to a uniformly elliptic equation in $\R^N$, so that  by \cite[Theorem 5]{moser}
there exists $\lim_{|x|\to \infty} (u_0)_{x_i}(x)=c_i$, and then, again by \cite[Theorem 4]{moser}, we conclude that $u_0$ is an affine function. 

\end{proof}

\section{Grim reapers}\label{grim}
In  the isotropic case there exists only one possible complete translating graph in $\R^2$, 
up to dilations   and translations,  which is called the grim reaper.   In particular this implies that there are not entire graphical translators. 
We will show that the same result holds  also in the anisotropic setting. 

Since we deal with complete but not entire translating graphs, we will not consider  the evolution of the subgraphs of the function (which is not well defined), but the evolution of the boundary $\partial E_0$, that is the graph of $u$, with normal vector  at every $(x, u(x))$ given by $(-u'(x), 1)$.
  
 We start with some technical lemmas, which hold in the regular case.   
 
 \begin{lemma} \label{lemlem}   
 Assume  \eqref{phiregular}.  Then there exist two constants $0<c_1\le c_2$ such that 
 \[
 \frac{c_1}{\phi(t,1)^3}\le \phi_{xx}(t,1)\le \frac{c_2}{\phi(t,1)^3}
 \qquad \forall t\in\R.
 \] 
 Moreover the function  \[h(t):= \phi(t,1)- \phi(1,0)|t|\]
 satisfies 
 \begin{enumerate}\item $h$ is convex and $\lim_{t\to\pm\infty}\frac{h(t)}{t}=0$;
 \item $h'(t)\leq 0$ for $t>0$,   $h'(t)\geq 0$ for $t<0$ and $0\geq h'_+(0)=\lim_{t\to 0^+}h'(t)\geq -2\phi(1,0)$,  whereas $0\leq h'_-(0)=\lim_{t\to 0^-}h'(t)\leq 2\phi(1,0)$;
 \item $-\phi(0,1)\leq h(t)\leq  h(0)=\phi(0,1)$ for all $t\in \mathbb{R}$. 
 \end{enumerate} 
  \end{lemma}
  
  \begin{proof}
 From  \eqref{phiregular} it follows that there exist $0<c_1\le c_2$ such that 
 \[
 c_1 \le {\rm det}\, \nabla^2 \left( \frac{\phi^2}{2}\right)(x,z)\le c_2 \qquad \forall (x,z)\in \R^2.
 \]
 Noting that $\phi(x,z)=|z|\phi(x/z,1)$ for all $(x,z)$ with $z\ne 0$, a direct computation shows that
 \[
 {\rm det}\, \nabla^2 \left( \frac{\phi^2}{2}\right)(x,z) = \phi\left(\frac xz,1\right)^3 \phi_{xx}\left(\frac xz,1\right),
 \] 
 which implies the first assertion. 
 
 Now, observe that by positive $1$-homogeneity $\phi(-1,0)=\phi(1,0)$, so we get  for $t\neq 0$, 
 \[\frac{h(t)}{t}=\frac{|t|}{t} \phi\left(\frac{t}{|t|}, \frac{1}{|t|}
 \right)-\frac{|t|}{t}\phi(1,0)\to 0 \qquad \text{as $t\to\pm\infty$.}\]
 Now for $t>0$, $h'(t)=\phi_x(t,1)-\phi(1,0)$. First of all we observe that by Remark \ref{norma}, $\phi_x(1,0)=\phi(1,0)$. 
Moreover    by the first part of the proof,  we have that $\phi_x(t,1)$ is a monotone increasing function, and moreover recalling Remark \ref{norma},  
 for $t\neq 0$, \[\phi_x(t,1)=\frac{t}{|t|} \phi_x\left(1, \frac{1}{t}\right)\to \begin{cases} \phi_x(1,0)& \text{ as $t\to +\infty$}\\ -\phi_x(1,0)& \text{ as $t\to -\infty$}.\end{cases}\] So, it follows that $h'(t)\leq 0$ for $t>0$ and $h'(t)\geq 0$ for $t<0$.  We get $h'_+(0)=\phi_x(0,1)-\phi_x(1,0)$. Now observe that by convexity $\phi(1,0)=\phi(-1,0)\geq \phi(0,1)-\phi_x(0,1)-\phi_z(0,1)=-\phi_x(0,1) $. So we conclude $h'_+(0)=\phi_x(0,1)-\phi_x(1,0)\geq -2\phi(1,0)$. The argument for $h'_-(0)$ is completely analogous. 
  
Finally, by convexity, recalling Remark \ref{norma} and the fact that $\phi_x(1,0)=\phi(1,0)$, we get for $t>0$
 \[h(t)=\phi(t,1)-\phi_x(1,0)t\geq \phi(1,0)+\phi_x(1,0)(t-1)+ \phi_z(1,0)-\phi_x(1,0)t=\phi_z(1,0)\] whereas for $t<0$ recalling that $\phi_x(-1,0)=-\phi_x(1,0)$, and $\phi(-1,0)=\phi(1,0)$,
 \[ h(t)=\phi(t,1)+\phi_x(1,0)t\geq \phi(-1,0)-\phi_x(1,0)(t+1)- \phi_z(1,0)+\phi_x(1,0)t=-\phi_z(1,0).\] 
 Again by convexity we conclude that 
 \[ \phi(0,-1)\geq \phi(1,0)- \phi_x(1,0)-\phi_z(1,0)=-\phi_z(1,0)\] and then $\phi_z(1,0)\geq -\phi(0,-1)=-\phi(0,1)$. On the other hand also \[ \phi(0,1)\geq \phi(1,0)- \phi_x(1,0)+\phi_z(1,0)=\phi_z(1,0)\] and therefore $-\phi_z(1,0)\geq -\phi(0,1)$. The two inequalities give the conclusion. 
 \end{proof}
 \begin{lemma}
\label{lemmaconvex}  Assume  \eqref{phiregular}. Let $I\subseteq \R$ be an open bounded interval and $u:I\to\R$ be a convex $C^2$ function such that 
$\lim_{x\to \partial I} u(x)=+\infty$ and 
\[\exists c>0\qquad \left(\phi_{x}(u', -1)\right)'<\frac{1}{c}\qquad \text{ for every $x\in I$.}\]   
Then for every $p'\in \{(x,z)\ | z\geq  u(x)\}$ such that $(x,u(x))\in\partial (p'+cW_{\phi^0})$ for some $x\in I$, there holds that  $p'+cW_{\phi^0}\subseteq \{(x,z)\ | z\geq u(x)\}$.  \end{lemma}
\begin{proof} We may assume without loss of generality that $c=1$, the other cases can be obtained by rescaling.   Let $F=\{(x,z)\ |\ z\geq u(x)\}$  be the epigraph of $u$. Then by definition
\[H_\phi((x,u(x)), F)=\left(\phi_{x}\left(\frac{u'}{\sqrt{1+(u')^2}},-\frac{1}{\sqrt{1+(u')^2}}\right)\right)'=\left(\phi_{x}(u', -1)\right)'\]
where the last equality comes from the fact that $\phi$ is a norm, so it is positively $1$-homogeneous, and then $\phi_x(tx,tz)=\frac{t}{|t|}\phi_x(x,z)$ for all $t\neq 0$. 
 Recalling \eqref{wcu}, and using the assumptions, we have that  $H_\phi(W_{\phi^0})=1> H_\phi((x,u(x)), F)$. 
 If $(x,u(x))\in \partial ( p'+W_{\phi^0})\cap \partial  F$, then the inequality on curvatures implies that there exists a neighborhood $U$  of $(x,u(x))$ such 
 that $p'+W_{\phi^0}\cap U\subseteq F\cap U$. 
Suppose by contradiction  that this inclusion is not satisfied  for $U=\R^2$. Therefore there exists an interval $(x-b', x+b)$ such that $(x-b',x+b) \subseteq U\cap I$ and  either $ (x+b, u(x+b)) \ \in\partial ( p'+W_{\phi^0})\cap \partial  F$ or $(x-b', u(x-b'))\ \in\partial ( p'+W_{\phi^0})\cap \partial  F$. Assume that the first case is verified, the other case is completely analogous. Since $W_{\phi^0}$ is a convex $C^2$ set, we may assume (eventually reducing $b'$), that $\partial ( p'+W_{\phi^0})\cap (x-b',x+b)\times \R$ coincides with the graph of a  $C^2$ convex function $w:(x-b',x+b)\to \R$ such that   $u\leq w$ for all $y\in (x-b',x+b)$. 

In particular we get  that $u(x)=w(x)$, $u'(x)=w'(x)$,  $u(x+b)=w(x+b)$,   and 
 \[\left(\phi_{x}\left(u',-1\right)\right)'< \left(\phi_{x}\left(w',-1\right)\right)' \qquad \text{ for all $y\in (x-b',x+b)$.}\]  
Integrating  the previous inequality between $x$ and $y\in (x,x+b)$, we get, 
\[ \phi_{x}\left( u'(y), -1 \right)<\phi_{x}\left( w'(y), -1 \right)\]
which  by Lemma \ref{lemlem} gives $u'(y)<w'(y)$ for all $y\in (x,x+b)$, which is in contradiction with the fact that $u(x)=w(x)$ and $u(x+b)=w(x+b)$. 
\end{proof} 

 We prove existence of complete translating graph in $\R^2$. We start with the case of regular anisotropies and then obtain the other cases by approximation. 
 \begin{theorem}\label{translator2} 
  There exists a complete  graphical translating solution to \eqref{mcf} which is   given    (up to dilations  and translations) by   $\partial E_t=\partial E+te_2$, where $\partial E$ solves \eqref{trans0}. In particular, $\partial E$ is the graph  of a convex  function $u:I\to \R$ , where $I$ is an interval. If  \eqref{phiregular} holds, then $u$ is characterized  as   the unique solution to 
\begin{equation} \label{ode}\begin{cases}  \psi\left(-u',1\right)\phi_{xx}\left(-u',1\right)u''=1   \\
u(0)= u'(0)=0. \end{cases}\end{equation} 
%Finally  other grim reapers are obtained in this way: we fix  a rotation $\rho:\R^2\to \R^2$, and we apply $\rho^{-1}$ to the grim reapers associated to $\psi\circ \rho$ and $\phi\circ \rho$. 

\end{theorem}  

\begin{proof}\ \
 We fix $c=1$ in \eqref{trans0}, since the case $c\neq 0$ can be obtained by dilations.
 
 We start considering the case in which \eqref{phiregular} holds and then the general case will be obtained by approximation.
 We observe that  the equation \eqref{trans0}  when $\partial E$ is the  graph of a function $u:\Omega\subseteq \R\to \R$ reads 
\[\psi\left(-u',1\right) \left(\phi_{x}\left(-u',1\right)\right)_x=-1\qquad x\in \Omega. \]
Up to translations we may also  assume that $u(0)=0 $.

  %$\left(\frac{-u'}{\sqrt{1+u'^2}},\frac{1}{\sqrt{1+u'^2}}\right)\in\mathbb{S}^1$, therefore  
The function defined as 
\[f(u'):=(1+u'^2)\, \psi\left(-u',1\right)\phi_{xx}\left(-u',1\right) \]
  is continuous, and moreover  by   Lemma \ref{lemlem},
 \[c_1  \psi\left(\frac{-u'}{\sqrt{1+u'^2}},\frac{1}{\sqrt{1+u'^2}}\right)\leq f(u')\phi^3\left( \frac{-u'}{\sqrt{1+u'^2}},\frac{1}{\sqrt{1+u'^2}}\right)\leq 
 c_2 \psi\left(\frac{-u'}{\sqrt{1+u'^2}},\frac{1}{\sqrt{1+u'^2}}\right) \] 
  so there exist two constants $0<\bar c_1\le \bar c_2$ such that $\bar c_1\leq f(u')\leq \bar c_2$. 

We assume that $u'(0)=0$ and we define $v(x)=u'(x)$. Then $v$ is a solution to 
\begin{equation}\label{ode1} \begin{cases} f(v(x))v'(x)=1+(v(x))^2\\ v(0)=0. \end{cases} \end{equation} 
Note that   $v(x)$   is defined in a   maximal  interval $I$ such that $ (-\bar c_1\pi, \bar c_1\pi)\subseteq I \subseteq (-\bar c_2\pi, \bar c_2\pi)$, that $v$ is strictly increasing  and that 
\[\tan \frac{x}{\bar c_2}  \leq v(x) \leq \tan \frac{x}{\bar c_1}.\]  
In particular  for every $\alpha$ the solution $v_\alpha$  to \eqref{ode1} with initial data $v_\alpha(0)=\alpha$ is obtained as $v_\alpha(x)=v(x+\beta)$ for some $\beta\in \R$. This implies that, up to translations, we may assume that $u'(0)=0$.

Finally we observe that the length of the maximal interval $I$ of existence for the solution to \eqref{ode} is actually  bounded independently of the constants $c_1,c_2$ appearing in Lemma \ref{lemlem}. Assume that $I=(a,b)$ and integrate \eqref{ode} in $(a,b)$, recalling that $I$ is the maximal interval of existence we obtain that 
\begin{equation}\label{integral}0\leq b-a=|I|=\int_{-\infty}^{+\infty} \psi(t,-1)\phi_{xx}(-t,1)dt\leq \max_{\mathbb{S}^1}\psi \int_{-\infty}^{+\infty} \sqrt{1+t^2}\phi_{xx}(-t,1)dt.\end{equation} 
 Recalling the definition and the properties of the function $h$ in Lemma \ref{lemlem} we get for $M> 0$, 
 \begin{eqnarray*}  \int_0^{M} \sqrt{1+t^2}h''(-t)dt &=&-\sqrt{1+M^2} h'(-M)+h'_{-}(0)-\int_{-M}^0\frac{t}{\sqrt{1+t^2}}h'(t)dt\\ &\leq& -\sqrt{1+M^2} h'(-M)+h'_{-}(0)+ h(0)- h(-M)\\  &\leq &\sqrt{1+M^2} h'(-M)+ 2\phi(1,0)+  2\phi(0,1) .\end{eqnarray*}  Now we observe that as $M\to +\infty$, $\sqrt{1+M^2} h'(-M)\to 0$ since $h'\geq 0 $ in $(-\infty, 0)$, is increasing and  $\int_{-\infty}^0 h'(t)dt<+\infty$. Therefore, in the previous inequality we obtain 
 \[\int_0^{+\infty}  \sqrt{1+t^2}\phi_{xx}(-t,1)dt\leq 2\phi(1,0)+ 2\phi(0,1) . \]
 With a completely analogous argument we get also that
  \[\int_{-\infty}^0 \sqrt{1+t^2}\phi_{xx}(-t,1)dt\leq 2\phi(1,0)+2\phi(0,1). \] This implies by \eqref{integral} that 
\begin{equation}\label{integral2}
|I|\leq  4(\phi(1,0)+ \phi(0,1) )\max_{\mathbb{S}^1}\psi.\end{equation} 
Assume now that \eqref{phiregular} does not hold. Let $\phi^n$ be a sequence of norms which satisfy \eqref{phiregular},  and such that $\phi^n\to \phi$ locally uniformly. 
By the previous arguments, for every $n$ we get a convex function $u^n$  which solves \eqref{ode} in an interval $I^n $. First of all we observe that by \eqref{integral2} the intervals $I^n$ are equibounded, and  converge in Hausdorff sense, up to subsequences to a limit interval $I$. 
 Moreover, since $\partial E^n$ is a solution to \eqref{trans0}, we get that  (recalling that the normal vector is $(-(u^n)', 1)$), 
\[H_{\phi^n}(p, E^n)= -\frac{1}{\sqrt{1+((u^n)')^2} \ \psi \left(\frac{-(u^n)',}{\sqrt{1+((u^n)')^2}},\frac{1}{\sqrt{1+((u^n)')^2}}\right)}
\geq -\frac{1}{\min_{\mathbb{S}^1} \psi}>-\frac{1}{c}\] where $0<c<\min_{\mathbb{S}^1} \psi$. 
Note that the set $F_n=\{(x,z) \ | \ x\in I^n, z\geq u^n(x)\}$ is a convex set, and moreover by the previous estimate, we get that at every 
$p\in \partial F_n$,  $H_{\phi^n}(p, F_n)<\frac{1}{c}$. In particular this implies by 
Lemma \ref{lemmaconvex} that  if $p'^n\in F_n$ is such that $p\in \partial(p'^n+c W_{(\phi^n)^0})$ (so in particular $|p-p'^n|\leq \text{diam }(c W_{(\phi^n)^0})$) 
then $p'^n+c\ W_{(\phi^n)^0}\subseteq F_n$.  Recalling that $\phi^n\to \phi$ locally uniformly, we get for arbitrary and now fixed $c'<c$ and  $n$ sufficiently large  
 \[p'^n+c' \ W_{\phi^0}\subseteq F_n  \] where $p'^n\in F_n$ is such that 
 $\partial(p'^n+c \ W_{(\phi^n)^0})\cap \partial F_n\neq \emptyset$.  In particular if we take $p'^n\in F_n$ such that $(0,0)\in \partial(p'^n+c \ W_{(\phi^n)^0})$, then, $p'^n\to p' $ up to subsequence and we get that for $n$ sufficiently large w.l.o.g. $p'+c' \ W_{\phi^0}\subseteq F_n  $.  
 
  Note that $F_n$ are epigraphs of convex functions such
  that $(0,0)\in \partial F_n$, and the previous estimates imply that the sequence of  convex sets $F_n$  are  contained in the strips $I^n\times [0,+\infty)$, which converge to $I\times [0, +\infty)$, and contain   $p'+c' \ W_{\phi^0}$. This implies that up to subsequences, the sets $F_n$ converge locally in Hausdorff sense to a convex set $F$, such that  $F$ is contained in the strip $I\times [0, +\infty)$, and $(0,0)\in \partial F$.  Moreover $F$ is the epigraph of a convex function $u$, and $u_n$ converges to $u$ locally uniformly. By passing to the level set formulation and using the stability with respect to locally uniform convergence of the distributional solutions to \eqref{pdelevel}, see \cite{cmnp2}, we get that $\partial F$ solves \eqref{trans0}.

  \end{proof}

\begin{remark}\upshape\label{remc1}   In the purely crystalline case, that is, when  $W_{\phi^0}$   is a convex  polygon in $\R^2$ and the mobility is the natural one, that is $\psi=\phi$,
we may construct directly a complete  translating solution  $\partial E_t =\partial E_0 + t e_2$, with a similar argument  as the one used in \cite{ggh} to construct self similar evolving crystals.   

  We fix $\nu_1, \dots, \nu_{k}$ as the ordered  set of adjacent normal orientations of   $\partial W_{\phi^0}$ with $\nu_i\cdot e_2<0$ and let 
 $\Delta(\nu_j)$  be the length of the edge of $W_{\phi^0}$ having $\nu_j$ as exterior normal.  
 
 We  construct  $\partial E_0$  as a  polygonal curve consisting  of a finite union of  segments  $S_1, \dots, S_{k}$ and two half-lines $S_0= (0, +\infty)e_2$, 
 $S_{k+1}=Le_1+(\hat L, +\infty)e_2$, where $L>0$ and $\hat L\in\R$ will be chosen later. In particular, $\partial E_0\cap (0,L)\times\R$ is
  the graph of a convex piecewise linear function $u_0: (0,L)\to \R$, with $u_0(0)=0$, $u_0(L)=\hat L$. 
For every $i\in \{1, \dots, k\}$ we require that $(-\nabla u_0(x), 1)/\sqrt{1+|\nabla u_0(x)|^2}=-\nu_i$ for all $x$ such that $(x, u_0(x))\in S_i$. 
Recalling Definition \ref{defanisotropy} (see also \cite{ggh, gp1, gp2})
  the crystalline curvature at every $p_i\in S_i$  is given by 
 \begin{equation}\label{cris}H_\phi(p_i, E_0)= -\frac{\Delta(\nu_i)}{L_i}\end{equation} 
 where $L_i$ is the length of the segment $S_i$.

%In order to fix the lengths  $L_i$, and thus also $L$ and $\hat L$). 
%We observe that to get that  $\partial E_0$ satisfies \eqref{trans0},  
Now we need to impose that the vertical speeds of the segments agree, i.e.,  
 \[ 
 \frac{\phi(\nu_j) H_\phi(p_j, E_0)}{\nu_j\cdot e_2}=\frac{\phi(\nu_{j-1}) H_\phi(p_{j-1}, E_0)}{\nu_{j-1}\cdot e_2}\,, \quad 2\le j\le k.\]
 Recalling  \eqref{cris},   we then get 
 \[
 L_j=  \frac{\phi(\nu_j)\Delta(\nu_j) \ \nu_{j-1}\cdot e_2}{\phi(\nu_{j-1})\Delta(\nu_{j-1})\ \nu_{j}\cdot e_2}\,L_{j-1}, \quad 2\le j\le k.
 \]
If we fix  $L_1=-\phi(\nu_1)\Delta(\nu_1)/( \nu_1\cdot e_2)$, by the previous equation  the other lengths satisfy 
$L_j=-\phi(\nu_j)\Delta(\nu_j)/(\nu_j\cdot e_2)$ for  all $2\le j\le k$. Therefore, recalling \eqref{cris}, we get that  $\partial E_0$ is a solution to  \eqref{trans} with $c=1$ and this implies that $\partial E_t= \partial E_0+te_2$. 
\end{remark} 

Finally, using the translating graphs obtained in  Theorem \ref{translator2}, we can construct grim reaper solutions, that is,
translating hypersurfaces in $\R^N$ asymptotic to  two parallel hyperplanes.

\begin{proposition} \label{dilded} 
For every $e\in \mathbb{S}^{N-1}$  there exists a complete graphical translator $\partial E$ for \eqref{mcf}  given by the graph  of a function $v_e(x) :=u_e(x\cdot e) $,  
where $u_e: I\to \R$ is a convex function and $I$ is an interval.  
\end{proposition} 

\begin{proof} Assume first that \eqref{phiregular} holds and define 
$\phi^e,  \psi^e:\R^2\to\R$ as the projections of $\phi, \psi$, that is $\phi^e(t ,z):=\phi(t e,z)$, $\psi^e(t,z):=\psi(te,z)$ for every $t\in\R, z\in \R$. Then, since  also $\phi^e$ satisfies \eqref{phiregular},  we can apply Theorem \ref{translator2} with $\phi^e, \psi^e$  and obtain the convex function 
$u_e: I\to \R$ as the unique solution to  \[\begin{cases}  \psi^e\left(-u',1\right)\phi^e_{xx}\left(-u',1\right)u''=1   \\
u(0)= u'(0)=0. \end{cases}\]  
We now define the function $v_e(x):=u_e(x\cdot e)$, and we observe that $\nabla v_e(x)=u_e'(x)e$. In particular, we have 
\[ 
-\psi(-\nabla v_e,1)\div(\nabla_x\phi(-\nabla v_e,1))= \psi^e\left(-u_e',1\right)\phi^e_{xx}\left(-u_e',1\right)u_e''=1\] 
and so $v_e$ is a solution to \eqref{trans}. This implies that  its graph is a complete translating hypersurface for \eqref{mcf}. 

 In the general case, we  proceed by approximation  as in Theorem \ref{translator2}.  
 \end{proof} 
 
As in \cite{hof}, up to a rotation of the coordinate system, from the solutions in Proposition \ref{dilded}  one can easily construct \emph{tilted} grim reapers.

\begin{proposition} \label{tilted} 
For every $e,t\in \mathbb{S}^{N-1}$ and $\lambda\in\R$  there exists a complete graphical translator $\partial E$ for \eqref{mcf}  given by the graph  of a function 
$v(x) :=u(x\cdot e) + \lambda x\cdot t$,  where $u: I\to \R$ is a convex function depending on $e,t,\lambda$ and $I$ is an interval.  
\end{proposition} 

%%%%%%%%%%%%%%%%%%%%%%%%%%%%%%%%%%%%%%%%%%%%%%%%%%%%%%%%%%
\section{Bowl solitons for cylindrical anisotropies and mobilities} \label{bowl} 
In this section we consider the case in which the mobility and the anisotropy satisfy the following assumption: 
There exist  two functions $F,G:[0, +\infty)\times[0, +\infty)\to [0, +\infty)$ and a norm $\xi:\R^N\to  [0, +\infty)$ such that 
\begin{equation}\label{cil} \phi(x,z)= F(\xi(x), z)\qquad\text{and}\qquad \psi(x,z)=G(\xi(x),z).
\end{equation} 
We  can extend $F,G$ to  the whole of $\R^2$ by letting $F(t,s)=F(-t,s)=F(t,-s)=F(-t,-s)$, and similarly for $G$.   
Note that, by the properties of $\phi, \psi$, the extended functions $F,G$ are norms on $\R^2$.

Under assumption \eqref{cil}, the Wulff shape associated to the anisotropy $\phi$ is cylindrical, in the sense that all the sections of the Wulff shape 
along the $e_{N+1}$-direction are homothetic.

\begin{proposition} \label{polar} 
Let $\phi:\R^{N+1}\to [0, +\infty)$ a norm which satisfies \eqref{cil}. Then $\phi^0(x,z)=F^0(\xi^0(x), z)$, where $F^0(t,s)=\max_{F(t',s')\leq 1} (tt'+ss')$ and 
$\xi^0(x)=\max_{\xi(x')\leq 1} x\cdot x'$. 
\end{proposition} 

\begin{proof} Fix $z\in \R$ and denote $F^{-1}_z$ the inverse of the function $t\in [0, +\infty) \to F(t,z)$. 
  \begin{eqnarray}\nonumber \phi^0(x,z)&= &\max_{F(\xi(x'), |z'|)\leq 1} (x\cdot x'+zz')= 
  \max_{|z'|\leq 1, \xi(x')\leq F^{-1}_{z'}(1)} (x\cdot x'+zz')\\ &=&\max_{|z'|\leq 1, \xi(h)\leq 1} (h\cdot x F^{-1}_{z'}(1)+zz')=
  \max_{|z'|\leq 1} (\xi^0(x) F^{-1}_{z'}(1)+zz'). \label{c1} \end{eqnarray}
  Now observe that, if $t,s\geq 0$, then
\[F^0(t,s)=\max_{F(t',s')\leq 1, t',s'\geq 0} (tt'+ss')=
\max_{0\leq s'\leq 1, 0\leq t'\leq F^{-1}_{s'} (1)} (tt'+ss')=  \max_{0\leq s'\leq 1} (t F^{-1}_{s'} (1)+ss'). \] 
Therefore, taking $t=\xi^0(x)$ and $s=z$, and  substituting in \eqref{c1} we get 
\[\phi^0(x,z)= F^0(\xi^0(x), z).\]
\end{proof}

Under assumption  \eqref{cil}, the  equation \eqref{trans} for graphical translators reads as follows: 
\begin{equation}\label{transcil} -\div\left( F_t(\xi(-\nabla u),1)\nabla \xi(-\nabla u)\right)=\frac{1}{G(\xi(-\nabla u),1)}.
\end{equation} 
We shall look for solutions having the same symmetries as the Wulff shape. Recalling from Proposition \ref{polar} that $\phi^0(x,z)= F^0(\xi^0(x),z)$,
 then  we look for solutions \[u(x)=v(\xi^0(x))\]  
 where $v:[0, \infty)\rightarrow \mathbb{R}$ is a convex function.
 
Recalling that for  $t\neq 0$, $\xi(tx)=|t|\xi(x)$,  we get $\nabla \xi(tx)=\frac{t}{|t|}\nabla \xi(x)$ and moreover since $\xi$ and $\xi^0$ are dual norms we get that $\xi(\nabla \xi^0(x))=1$ and $\xi^0(x)\nabla\xi(\nabla\xi^0(x))=x$.  For more details we refer to \cite[Section 2.1]{bp}. This implies that
\begin{eqnarray*} \xi(-\nabla u(x))&=&|v'(\xi^0(x))|\xi(\nabla\xi^0(x)))= |v'(\xi^0(x))|\\  \nabla \xi(-\nabla u(x))&=& -\frac{v'(\xi^0(x))}{|v'(\xi^0(x))|}\nabla \xi(\nabla\xi^0(x))= -\frac{v'(\xi^0(x))}{|v'(\xi^0(x))|}\frac{x}{\xi^0(x)}.\end{eqnarray*} 

We substitute these formulas in \eqref{transcil} and obtain 
\[\div\left( F_t(|v'(\xi^0(x))|,1)\frac{v'(\xi^0(x))}{|v'(\xi^0(x))|}\frac{x}{\xi^0(x)}\right)=\frac{1}{G(|v'(\xi^0(x))|, 1)}.\]
We compute the divergence,  recalling that $\div(\frac{x}{\xi^0(x)})=\frac{N-1}{\xi^0(x)}$ and that $\nabla\xi^0(x)\cdot x=\xi^0(x)$,  so that  
\[F_{tt}(|v'(\xi^0(x))|,1) v''(\xi^0(x))+F_t(|v'(\xi^0(x))|,1)\frac{v'(\xi^0(x))}{|v'(\xi^0(x))|}\frac{N-1}{\xi^0(x)}
=\frac{1}{G(|v'(\xi^0(x))|, 1)}.\]
and then, letting $r=\xi^0(x)$ and $w=v'$, we get
\begin{equation}\label{w1} w'(r)= \frac{1}{F_{tt}(|w(r)|,1)} \left(\frac{1}{G(|w(r)|,1) }-\frac{N-1}{r} \frac{w(r)}{|w(r)|}F_t(|w(r)|, 1)\right). \end{equation} 
If  $F(t,s)=G(t,s)=\sqrt{t^2+s^2}$ and $\xi(x)=|x|$,  we get exactly $w'= \left(1-\frac{N-1}{r}w\right)(1+w^2)$ which is the equation for radially symmetric graphical translators. 
   
Note that by Lemma \ref{lemlem}, if $F^2$ is uniformly convex and $F\in C^2(\R^2\setminus 0)$, there exist $0<c_1\leq c_2$ such that 
\begin{equation}\label{lemlem1} \frac{c_1}{F^3(\alpha,1)}\leq F_{tt}(\alpha,1)\leq \frac{c_2}{F^3(\alpha,1)}\qquad \forall \alpha\in \R.\end{equation}
Moreover since $F(\alpha,1)$ is an even convex function, if $F\in C^2(\R^2\setminus 0)$, we get that necessarily 
\begin{equation}\label{ft} F_t(0,1)=0\qquad \text{and } F_t(\alpha,1)>0 \ \ \forall \alpha>0.\end{equation}

 \begin{lemma} \label{lemma1} 
 Assume \eqref{cil}, with $F^2, G^2$ uniformly convex , and  $F,G\in C^2(\R^2\setminus \{0\})$.
 Then there exists  $w\in C^1(0, +\infty)$, which is positive, increasing and solves  
 \begin{equation}\label{sys1} \begin{cases} w'= \dfrac{F_t(w(r), 1)}{F_{tt}(w(r),1)} \left(\frac{1}{G(w(r),1)F_t(w(r), 1)}-\frac{N-1}{r} \right)\\ 
\lim_{r\to 0^+} w(r)=0. 
\end{cases} \end{equation} 
Moreover  $\lim_{r\to +\infty} \frac{w(r)}{r}=\frac{1}{(N-1) G(1,0)F(1,0)}$. 
 \end{lemma}
 \begin{proof}  
    We define for $\alpha>0$,  the function   \begin{equation}\label{f} f(\alpha):= \frac{1}{G(\alpha,1)F_t(\alpha, 1)}\qquad f(\cdot):(0, +\infty)\to (0, +\infty)\end{equation} and observe, recalling \eqref{lemlem1} and  \eqref{ft} (which holds also for $G$),  that it is strictly  decreasing in $(0, +\infty)$, moreover that $\lim_{\alpha\to 0^+}f(\alpha)=+\infty$, since $F_t(0,1)=0$ and  $\lim_{\alpha\to +\infty}f(\alpha)=0$ by positive $1$-homogeneity of $G$ and positive $0$-homogeneity of $F_t$. 
 So the equation $f(\alpha)=\frac{N-1}{r}$ admits a unique  positive solution \begin{equation}\label{a} \alpha(r):=f^{-1}\left(\frac{N-1}{r}\right)\qquad \alpha(\cdot):(0, +\infty)\to (0, +\infty).\end{equation} 
It is easy to check  that  $\alpha(r)$ is strictly increasing, that  $\lim_{r\to 0^+}\alpha(r)= 0$, whereas  $\lim_{r\to +\infty}\alpha(r)= +\infty$. 
By $1$-homogeneity of $G$ and $0$-homogeneity of $F_t$, we get   \[\lim_{\alpha\to +\infty} \alpha f(\alpha) =\lim_{\alpha\to +\infty}\frac{ \alpha}{ \alpha
G(1,1/\alpha)F_t(1, 1/\alpha)}=\frac{1}{G(1,0)F_t(1,0)}.\]  
 Therefore 
 \begin{equation}\label{as}\lim_{r\to +\infty} \frac{\alpha(r)}{r}=\lim_{\alpha\to +\infty} \frac{\alpha f(\alpha)}{N-1}= \frac{1}{(N-1)G(1,0)F_t(1,0)}.\end{equation} Finally,  since  $F(t,s)= tF_t(t,s)+sF_s(t,s)$, we get that $F_t(1,0)= F(1,0)$. 
 
As long as $w(r)>0$, \eqref{w1} can be written as \[w'(r)=  \frac{F_t(w(r), 1)}{F_{tt}(w(r), 1)}\left(f(w(r))-f(\alpha(r))\right).  \]
Note that if  $0<w(r)<\alpha(r)$, (which is equivalent to $f(w)-\frac{N-1}{r}>0$) we get that $w'(r)>0$, whereas  if $w(r)>\alpha(r)$, then $w'(r)<0$. 
 This implies that if $w$ solves the ode in some interval $(\rho, \rho+s)$ for some $\rho>0$ and   $0<w(\rho)<\alpha(\rho)$, then   $0<w(r)\leq \alpha(r)$ for all $r>\rho$, since  in the region $w>\alpha$, we would get  $w'<0$. Then we get a solution $w$  defined for all $r>\rho$, which is positive and increasing.

 We fix $\rho>0$ and consider the system 
\[ \begin{cases} w'(r)=\frac{1}{F_{tt}(w(r),1)} \left(\frac{1}{G(w(r),1) }-\frac{N-1}{r}  F_t(w(r), 1)\right)\\
w(\rho)=  \frac{\alpha(\rho)}{2}. 
\end{cases} \]
Note that, by the previous discussion, the  system admits a unique solution $w_\rho$
which satisfies $0<w_\rho(r)\leq \alpha(r)$, for all $r>\rho$, and then is defined for all $r>\rho$ and is strictly increasing. Moreover, $w_\rho\sim r$ as $r\to +\infty$.

   We define $w(r)=\lim_{\rho\to 0^+} w_\rho(r)$. We get that the limit is locally uniform in $C^1$, by Arzel\'a-Ascoli Theorem, and moreover $w$ is a solution to
    \eqref{sys1}   which is positive, strictly increasing,  and satisfies $w(r)\leq \alpha(r)$. 
    
 Finally observe that $w(r)\to +\infty$ as $r\to+\infty$ and  by \eqref{lemlem1}, we get that
 \[w'(r)\geq \frac{1}{c_1} F_t(w(r),1)F^3(w(r),1)(f(w(r))-f(\alpha(r))).\] So necessarily $f(w(r))-f(\alpha(r))\to 0$ as $r\to +\infty$, since otherwise we would get $w'(r)\geq kw^3(r)$ for $r\geq r_0$ with suitable $r_0, k>0$, in contradiction with the fact that $w(r)$ is defined for all $r>0$. This implies that as $r\to +\infty$, $w(r)-\alpha(r)\to 0$, which gives the  desired asymptotic behavior. 
     \end{proof} 
\begin{theorem} \label{translator} 
Assume   \eqref{cil}.
Then there exists an entire function $u:\R^N\to \R$ whose graph is a
translating solution to \eqref{mcf}. Moreover, $u$ is convex and satisfies 
\[u(x) =  \frac{\xi^0(x)^2}{2(N-1)G(1,0)F(1,0)} +o(\xi^0(x)^2)\qquad \text{ for $|x|\to +\infty$.}\]  
Finally, if \eqref{phiregular} holds, the $u$ is unique up to dilations and translations. 
\end{theorem} 
\begin{proof} We start considering the case in which  $F^2,G^2$ are uniformly convex and $F,G\in C^2(\R^2\setminus \{0\})$. The general case will be obtained by approximation. 

Let $w$ the function constructed in Lemma \ref{lemma1}, and define  $u(x):=\int_0^{\xi^0(x)} w(r)dr$.  Then, $u$ is a solution to \eqref{transcil}, and moreover $u$ is  symmetric with respect to $\xi^0$, convex, since $w$ is increasing, and has quadratic growth at infinity. 

Now consider $F,G$  generic, and define a sequence of norms $F^n$, $G^n$ such that $(F^n)^2$, $(G^n)^2$ are uniformly convex, $F^n, G^n\in C^2(\mathbb{R}^2\setminus \{0\})$ and finally  $F^n\to F$, $G^n\to G$ locally uniformly.  We associate to every anisotropy $F^n $, with mobility $G^n$,  a solution $u^n$ to \eqref{transcil} as constructed above.

Since $F$ is a norm, we get that $F(t,s)= tF_t(t,s)+sF_s(t,s)$ (where $(F_t(t,s), F_s(t,s))$ is  the subdifferential of $F$ at $(t,s)$). 
Moreover  $F$ is positively $1$-homogeneous, whereas $F_t, F_s$ are positive $0$-homogeneous, so for $\alpha>0$,  \[F_t(\alpha,1)= \frac{F(\alpha,1)- F_s(\alpha,1)}{\alpha}=
F\left(1,\frac{1}{\alpha}\right)-\frac{1}{\alpha}F_s\left(\frac{\alpha}{\sqrt{1+\alpha^2}}, \frac{1}{\sqrt{1+\alpha^2}}\right)\]
and then $F_t(\alpha, 1)=F_t\left(\frac{\alpha}{\sqrt{1+\alpha^2}}, \frac{1}{\sqrt{1+\alpha^2}}\right)\to F(1,0)>0$ as $\alpha\to +\infty$.

Therefore, by locally  uniform convergence of the convex functions $F^n$ to $F$, there exist $\alpha_0$ and $n_0$  such that $F_t^n(\alpha,1)=F_t^n\left(\frac{\alpha}{\sqrt{1+\alpha^2}}, \frac{1}{\sqrt{1+\alpha^2}}\right)>\frac{F(1,0)}{2}>0$ for all $n\geq n_0$ and $\alpha\geq \alpha_0$. Eventually enlarging  $n_0, \alpha_0$ we get that also $ G^n\left(\frac{\alpha}{\sqrt{1+\alpha^2}}, \frac{1}{\sqrt{1+\alpha^2}}\right)>\frac{G(1,0)}{2}>0$ for all $n\geq n_0$ and $\alpha\geq \alpha_0$.
Recalling the definition \eqref{f}, and using the $1$-homogeneity and the locally uniform convergence, this implies that there exists $C>0$ and $\alpha_0, n_0$ such that for all $\alpha\geq \alpha_0$ and $n\geq n_0$  there holds 
\[f^n(\alpha)=\frac{1}{G^n(\alpha,1)F_t^n(\alpha,1)}=\frac{1}{\sqrt{1+\alpha^2}}\frac{1}{G^n\left(\frac{\alpha}{\sqrt{1+\alpha^2}}, \frac{1}{\sqrt{1+\alpha^2}}\right)F_t^n(\alpha,1)}\leq  \frac{4}{\alpha G(1,0) F(1,0)}. \]
This implies, recalling that $f^n$ are strictly decreasing,  
\[(f^n)^{-1}\left(\frac{4}{\alpha G(1,0)F(1,0)}\right)\leq \alpha\qquad \forall \alpha\geq \alpha_0, n\geq n_0.\]
Recalling the definition of $\alpha^n(r)$ in \eqref{a}, there holds
\[\alpha^n(r)=(f^n)^{-1}\left(\frac{N-1}{r}\right)\leq \frac{4}{(N-1)G(1,0)F(1,0)} r\qquad \forall r\geq \frac{(N-1)G(1,0)F(1,0)}{4} \alpha_0,\  n\geq n_0.\]
Therefore, since by Lemma \ref{lemma1} we get that $w^n(r)\leq \alpha^n(r)$, this implies that 
\[0\leq u^n(x)=\int_0^{\xi^0(x)} w^n(r)dx\leq  \int_0^{\xi^0(x)} \frac{4 r}{(N-1)G(1,0)F(1,0)} dr= \frac{2(\xi^0(x))^2}{(N-1)G(1,0)F(1,0)} . \]
Therefore, up to passing to a subsequence, recalling that $u_n$ are convex functions, we get that $u_n\to u$ locally uniformly. Therefore by stability properties with respect to uniform convergence of solutions to \eqref{pdelevel1}, we get that $u(x)+t$ solves \eqref{pdelevel1}, and then $u$ is a solution to \eqref{transcil}. 

Finally, if \eqref{phiregular} holds, we prove uniqueness of the solution $u$ constructed as above. Assume there exists $u_1(x)=\int_0^{\xi^0(x)} w_1(r)dr$, where $w_1(r)$ is another solution to \eqref{sys1} different from $w$. Then by uniqueness of  solution to the ode in \eqref{sys1}, we get that necessarily either $w(r)\leq w_1(r)$ or $w_1(r)\leq w(r)$ for all $r$. Assume that the first inequality is true. This implies that $u(x)\leq  u_1(x)$ are both solutions to \eqref{transcil} and  we may assume, up to adding a constant, that $u(x)\leq u_1(x)$ and $u(x_0)=u_1(x_0)$ for some $x_0$. So, for strong maximum principle, there holds $u_1=u$. 
\end{proof} 
 
\begin{remark}\upshape\label{remc2}
For the crystalline cases in which  $F(t,s)=|t|+|s|$ or $F(t,s)=\max  (|t|,|s|)$,  with natural mobility $F=G$, we may describe explicitly the shape of the translating solutions constructed in Theorem \ref{translator}.

In the first case, $\phi(x,z)=\xi(x)+|z|$, and then, recalling Proposition \ref{polar}, we get that  the Wulff shape is the cylinder $W_{\phi^0}=W_{\xi^0}\times [-1,1]$, where $W_{\xi^0}$ is the Wulff shape associated to the norm $\xi$.  
In this case the system \eqref{sys1} reads:  for some $r_0>0$, to be   appropriately chosen 
\begin{equation}\label{sys2} \begin{cases} \div\left(\frac{x}{\xi^0(x)}\right) = \frac{1}{1+w(\xi^0(x))} & \xi^0(x)>r_0\\ 
w(\xi^0(x))=0& \xi^0(x)\in [0, r_0].\end{cases}\end{equation} 
Recalling that  $\div\left(\frac{x}{\xi^0(x)}\right)=\frac{N-1}{\xi^0(x)}$, this gives that $w(r)=\frac{r}{N-1}-1$ for $r\geq r_0$, 
and then  \[u(x)=\begin{cases} \int_{r_0}^{\xi^0(x)} \left(\frac{r}{(N-1)}-1\right)dr= \frac{\xi^0(x)^2-r_0^2}{2(N-1)}-\xi^0(x)+r_ 
0& \text{ for }\xi^0(x)>r_0\\ 0 & \text{ for }0\leq\xi^0(x)\leq r_0.\end{cases} \] 
Now, the constant  $r_0$ has to be chosen in order to have that  the subgraph of $u$, that is $E=\{(x,z)\ | \ z\leq  u(x)\}$ solves \eqref{trans0} with $c=1$. For $p=(x,z)\in \partial E$ with $z>0$ the fact that \eqref{trans0} is satisfied is a consequence of the construction of the function $u$, using the solution to system \eqref{sys2}. So it is sufficient to choose $r_0$ such that  \eqref{trans0} is verified at  every $p=(x,0)\in \partial E$, so with  $\xi^0(x)\leq r_0$. Recalling that in this case   $\nu(p)=e_{N+1}$ and  also that $\phi=\psi$, \eqref{trans0} reads 
\[H_\phi(p, E)= -1.\]
We denote  $F_0:=\{x\in \R^N\ | \ \xi^0(x)\leq r_0\}$ and we have, by definition of crystalline mean curvature, recalling that the  Wulff shape is $W_{\phi^0}=W_{\xi^0}\times [-1,1]$,
\[H_\phi(p, E)=\frac{1}{|F_0|}\int_{F_0}  \div_\tau(\partial \phi(\nu(p))) dp=-\frac{1}{|F_0|}\int_{\partial F_0} \phi(\nu(p))dH^{N-1}(p)=-\frac{\Per_\phi (F_0)}{|F_0|}.\]
So, the condition on $r_0$ is that $\Per_\phi \{x\in \R^N\ | \ \xi^0(x)\leq r_0\}=|\{x\in \R^N\ | \ \xi^0(x)\leq r_0\}|$.

In the second case  $\phi(x,z)=\max(\xi(x),|z|)$, and then, recalling Proposition \ref{polar}, we get that  the Wulff shape is the double cone $W_{\phi^0}=\cup_{|z|\leq 1} ((1-|z|)W_{\xi^0}\times\{z\})$.   
In this case the system \eqref{sys1} reads:  for some $r_0>0$, to be   appropriately chosen (see below),  
\begin{equation}\label{sys3} \begin{cases} \div\left(\frac{x}{\xi^0(x)}\right) = \frac{1}{w(\xi^0(x))} &   \xi^0(x)>r_0\\ 
w(\xi^0(x))=1 & 0< \xi^0(x)<r_0. \end{cases}\end{equation} 
Arguing as before  we get that $w(r)= \frac{r}{N-1}$, for $r>r_0$, so  necessarily $r_0\geq N-1$ and the function $u$ is given by
\[ u(x)=\begin{cases} \int_{r_0}^{\xi^0(x)}  \frac{r}{N-1} dr+r_0= \frac{\xi^0(x)^2-r_0^2}{2(N-1)}+ r_0
& \text{ for }\xi^0(x)>r_0\\  \xi_0(x) & \text{ for }0\leq\xi^0(x)\leq r_0.\end{cases} \] For $p=(x,z)\in \partial E$ with $z>\xi^0(x)$ the fact that \eqref{trans0} is satisfied is a consequence of the construction of the function $u$, using the solution to system \eqref{sys3}. 
At $p=(x, \xi^0(x))\in \partial E$, there holds that $\nu(p)=\frac{(-\nabla \xi^0(x),1)}{\sqrt{1+|\nabla \xi^0(x)|^2}}$, and $\phi(\nu(p))=
\frac{1}{\sqrt{1+|\nabla \xi^0|^2}} \max (\xi(\nabla\xi^0(x)),1)=\frac{1}{\sqrt{1+|\nabla \xi^0|^2}}$.  So \eqref{trans0} reads
\[H_\phi(p, E)= -1.\] 
Observe that  $  \partial E\cap\{z\leq \xi^0(x)\}$ coincides with half of $r_0\partial W_{\phi^0}(x)$. Recalling that $H_{\phi} (r_0W_{\phi^0})=\frac{N}{r_0}$, we get that $r_0=N$. \end{remark}

%%%%%%%%%%%%%%%%%%%%%%%%%%%%%%%%%%%%%%%%%%%%%%%%%%%%%%%%%%%%%%%%%%%%% 

  \end{document}